\documentclass[a4paper, 11pt]{article}
\usepackage{latexsym,amssymb}
\usepackage[english]{babel}
\usepackage{bm}
\usepackage{amsfonts, amsthm}
\usepackage{amsmath}
\usepackage{mathrsfs}
\usepackage{multirow}
\usepackage{bm}
\usepackage{pgf,tikz}
\usepackage{rotating}
\usepackage{calculator}

\usepackage{xcolor}

\newcommand{\comment}[1]{}

\newcommand{\Keywords}[1]{\par\noindent{\bfseries Keywords}: #1}
\newcommand{\MSC}[1]{\par\noindent{\bfseries MSC}: #1}

\usepackage{amsthm}
\theoremstyle{plain}

\newtheorem*{Theorem*}{Theorem 1}

\newtheorem*{Lemma*}{Lemma}
\newtheorem{Definition}{Definition}[section]

\newtheorem{Theorem}[Definition]{Theorem}
\newtheorem{Proposition}[Definition]{Proposition}

\newtheorem{Lemma}[Definition]{Lemma}

\newtheorem{Corollary}[Definition]{Corollary}

\usepackage{dsfont}

\title{
Vertex-regular $1$-factorizations in infinite graphs
}
\author{Simone Costa\thanks{DII/DICATAM - Sez. Matematica, Universit\`a degli Studi di Brescia, Via
Branze 38, I-25123 Brescia, Italy. email: simone.costa@unibs.it}\ \
and Tommaso Traetta%
\thanks{
DICATAM - Sez. Matematica, Universit\`a degli Studi di Brescia,
Via Valotti 9, I-25123 Brescia, Italy. email: tommaso.traetta@unibs.it}}

\begin{document}
\maketitle
\begin{abstract}
The existence of $1$-factorizations of an infinite complete equipartite graph $K_m[n]$ (with $m$ parts of size $n$) admitting a vertex-regular automorphism group $G$ is known only when $n=1$ and $m$ is countable
(that is, for countable complete graphs) and, in addition, $G$ is a finitely generated abelian group $G$ of order $m$.

In this paper, we show that a vertex-regular $1$-factorization of $K_m[n]$ under the group $G$ exists if and only if $G$ has a subgroup $H$ of order $n$ whose index in $G$ is $m$.
Furthermore, we provide a sufficient condition for
an infinite Cayley graph to have a regular $1$-factorization. 
Finally, we construct 1-factorizations that contain a given subfactorization, both having 
a vertex-regular automorphism group.

\end{abstract}
\MSC{05C63, 05C70}

\section{Introduction}
In this paper, we deal with graphs, finite or infinite, which are simple and with no loops. Given a graph $\Lambda$, we denote by $V(\Lambda)$ the set of vertices,
and by $E(\Lambda)$ the set of edges of $\Lambda$, and refer to the cardinality
$|V(\Lambda)|$ of $V(\Lambda)$ as the order of $\Lambda$.
As usual, we use the notation $K_V$ for the complete graph whose vertex set is $V$.
Also, we denote by $K_m[n]$ the complete equipartite graph on $mn$ vertices partitioned into $m$ parts of size $n$; two vertices of $K_m[n]$ are adjacent if and only if they belong to distinct parts. Clearly $K_m[1]$ is isomorphic to
the complete graph $K_m$ of order $m$.

A graph is $r$-regular
if each of its vertices has $r$ edges incident with it.
A $1$-regular graph is then the vertex disjoint union of edges.
A subgraph $\Gamma$ of $\Lambda$ such that $V(\Gamma)=V(\Lambda)$ is called
a factor of $\Lambda$; equivalently, a factor of $\Lambda$ is a subgraph obtained by edge deletions only. A $1$-regular factor is simply called a $1$-factor.

A decomposition of $\Lambda$ is a set $\mathcal{G}=\{\Gamma_1,\ldots, \Gamma_n\}$ of subgraphs of $\Lambda$ whose edge-sets partition $E(\Lambda)$.
%If each $\Gamma_i$ is isomorphic to $\Gamma$, then we speak of a $\Gamma$-decomposition of $\Lambda$.
If each $\Gamma_i$ is a factor (resp. $1$-factor), we speak of a factorization (resp. $1$-factorization) of $\Lambda$.

An automorphism of $\Lambda$ is a bijection $\alpha$ of $V(\Lambda)$ such that
$\alpha(\Lambda)=\Lambda$, where $\alpha(\Lambda)$ is the graph obtained from $\Lambda$ by replacing each vertex,
say $x$, with $\alpha(x)$.
An automorphism of $\mathcal{G}$ is a bijection $\beta$ of $V(\Lambda)$ such that
$\beta(\Gamma)\in \mathcal{G}$ for every $\Gamma\in\mathcal{G}$.
It follows that $\beta(\Lambda)=\Lambda$, hence $\beta$ is necessarily an automorphism of the graph $\Lambda$.
If $\Lambda$ (resp. $\mathcal{G}$) has an automorphism group isomorphic to $G$ that acts sharply transitively on $V(\Lambda)$, we say that $\Lambda$ (resp. $\mathcal{G}$) is vertex-regular under $G$,
or simply $G$-regular.

Although regular $1$-factorizations have been widely studied for finite complete graphs \cite{BoLa, BoRi, Bonvi, Bonvi2, B01, HR, Ko, PaPe, Ri}, very little is known in the infinite case.
%Very little is known on regular $1$-factorizations of infinite complete (equipartite) graphs.
%In this paper we focus on regular $1$-factorizations of infinite complete (equipartite) graphs,
%Their existence has been widely investigated in the finite case,
%(see, for instance, \cite{HR,B01,BoLa,PaPe}
%\textcolor{red}
%{
%(se proprio vogliamo parlare del caso finito, dovremo allora citare molto di piu')
%})
%however much less is known in the infinite case.
In \cite{BoMa10}, the authors construct a $G$-regular 1-factorization of a countable complete graph for every finitely generated abelian infinite group $G$.
A complementary result has been obtained in \cite{Costa20}
where it is shown that there exists a $G$-regular $1$-factorization of the complete graph $K_G$ for every infinite group $G$ -- not necessarily countable -- with no involutions (i.e., elements of order $2$). It is worth pointing out that \cite{Costa20} provides a much more general result
concerning the existence of a $G$-regular factorization of complete graphs of every infinite order.
As far as we know, there is no other paper dealing with vertex-regular 1-factorizations of infinite graphs.

In this paper, we completely characterize the existence of $G$-regular $1$-factorizations of the complete equipartite graph $K_m[n]$,
for every infinite group $G$. More precisely, we prove the following.
\begin{Theorem}\label{main}
Let $G$ be an infinite group. There exists a $G$-regular $1$-factorization of
$K_m[n]$ if and only if $G$ has a subgroup $H$ of size $n$ whose index in $G$ is $m$.
\end{Theorem}
As a consequence, we obtain the existence of a $G$-regular $1$-factorization of
$K_G$ for every infinite group $G$.

More generally, if $\Lambda$ has a $G$-regular $1$-factorization, then $\Lambda$ is $G$-regular itself.
In other words, $\Lambda$ is necessarily a Cayley graph on $G$, and in Section 2 we recall some basic notions and well-known results on Cayley graphs and vertex-regular decompositions.

In Section 3 we provide a sufficient condition for an infinite Cayley graph on $G$ to have
a $G$-regular $1$-factorization: Theorem \ref{main1}.
Complete (equipartite) graphs are Cayley graphs and in Section 4
we essentially prove that they satisfy the assumption of Theorem \ref{main1}, thus proving Theorem \ref{main}. Finally, in Section 5, we construct vertex-regular 1-factorizations with given regular subfactorizations.

\section{Preliminary notions}
In this section, we recall some known facts on Cayley graphs and graph decompositions with a regular automorphism group on the vertex set. Throughout the paper, we denote groups in additive notation, even though they are not necessarily abelian.

Given a group $G$, a subset $S$ of $G\setminus \{0\}$ such that $S=-S$ is called a connection set.
The Cayley graph on $G$ with connection set $S$ is the simple graph
$Cay[G:S]$ having $G$ as vertex set and such that two vertices $x$ and $y$ are adjacent if and only if $x-y\in S$.
Note that, if $S=G\setminus \{0\}$ then $Cay[G:S]$ is the complete graph whose vertex set is $G$. More generally, if $H$ is a subgroup of $G$ of index $m$ and order $n$, then $Cay[G:G\setminus H]$ is isomorphic to $K_m[n]$; indeed,
two vertices are adjacent if and only if they belong to distinct right cosets of $H$, which therefore represent the $m$ parts of $Cay[G:S]$, each of size $n$.

Given a graph $\Gamma$ with vertices in $G$, the right translate of $\Gamma$ by an element $g\in G$ is the graph $\Gamma+g$ obtained from $\Gamma$ by replacing each of its vertices, say $x$, with $x+g$.
If $\Gamma$ is a Cayley graph on $G$, then $\Gamma + g = \Gamma$. Hence, the group of right translations of $\Gamma$ (which we recall to be isomorphic to $G$) is a vertex-regular automorphism group of $\Gamma$.
The following theorem shows that the graphs with a vertex-regular automorphism group are exactly the Cayley graphs.

\begin{Theorem}[\cite{Sab}]\label{Sabidussi}
A graph $\Lambda$ is $G$-regular if and only if
$\Lambda$ is isomorphic to a Cayley graph on $G$.
\end{Theorem}

Recalling that an automorphism of a decomposition of $\Lambda$ is also an automorphism of the graph $\Lambda$,
we have the following corollary.

\begin{Corollary}\label{Sabidussi:2}
If there exists a $G$-regular decomposition of $\Lambda$, then $\Lambda$ is isomorphic to a Cayley graph on $G$.
\end{Corollary}

Therefore, we focus on regular $1$-factorizations of Cayley graphs and recall two efficient methods to construct them.
%Now we want to investigate when a given graph $\Lambda$ admits a $G$-regular $\Gamma$-decomposition. Note that this condition is more strict than just requiring that $\Lambda$ has $G$ as a regular automorphism group. For this purpose, we need to set some notation.

Let $\Lambda = Cay[G : S]$, let $\{S_1,\ldots, S_n\}$ be a partition of $S$ into connection sets (i.e. $S_i=-S_i$)
and set $\Gamma_i=Cay[G:S_i]$, for every $i=1,\ldots,n$. Clearly, $\mathcal{G} = \{\Gamma_i\mid i=1,\ldots,n\}$ is
a decomposition of $\Lambda$. Also, each $\Gamma_i$ is a factor of $\Lambda$ and
it is fixed by the right translations induced by the elements of $G$. Therefore,
$\mathcal{G}$ is a $G$-regular factorization of $\Lambda$. If $S=\{s_1, \ldots, s_n\}$ contains only involutions and
$S_i=\{s_i\}$ for every $i=1,\ldots,n$, then each $\Gamma_i$ is a $1$-factor of $\Lambda$ and
$\mathcal{G}$ is a $G$-regular $1$-factorization of $\Lambda$. Denoting by $I(G)$ the set of all involutions of $G$, we then have the following.

\begin{Proposition}\label{involutions}
$Cay[G:S]$ has a $G$-regular $1$-factorization whenever $S\subseteq I(G)$.
\end{Proposition}

Another way of constructing $G$-regular $1$-factorizations relies on the concept of difference family.
Given a graph $\Gamma$ with vertices in $G$, the list of differences of $\Gamma$ is the multiset
$\Delta(\Gamma)$ of the differences $x-y$ and $y-x$ between every two adjacent vertices $x$ and $y$ of $\Gamma$.

\begin{Proposition}\label{basefactor}
Let $\Gamma$ be a $1$-factor of $K_G$. If each element of $\Delta\Gamma$ has multiplicity 1,
then $\{\Gamma+g\mid g\in G\}$ is a $G$-regular $1$-factorization of $Cay[G:\Delta\Gamma]$.
\end{Proposition}

Note that if $\Delta \Gamma$ contains an involution, then it must appear with even multiplicity.

\section{1-Factorizations of infinite Cayley graphs}
In this section, we provide a sufficient condition for a Cayley graph on the group $G$ to have a $G$-regular $1$-factorization. More precisely, we prove the following.
We recall that $I(G)$ denotes the set of all involutions of $G$.

\begin{Theorem}\label{main1}
Let $G$ be an infinite group and let $\Gamma=Cay[G:S]$.
If $|S\setminus I(G)|$ is either $0$ or $|G|$, then there exists a $G$-regular $1$-factorization of $Cay[G:S]$.
\end{Theorem}

Considering that $Cay[G:S\cap I(G)]$ has a $G$-regular $1$-factorization $\mathcal{G}_1$ by Proposition \ref{involutions},
it is enough to prove that $Cay[G:S\setminus I(G)]$ has a $G$-regular $1$-factorization $\mathcal{G}_2$. Indeed,
one can easily check that $\mathcal{G}_1\,\cup\,\mathcal{G}_2$ is a $G$-regular $1$-factorization
of $\Gamma=Cay[G:S]$.
Our strategy is to construct, by transfinite induction, a $1$-factor $\Gamma$ such that
$\Delta\Gamma = S\setminus I(G)$ whenever $|S\setminus I(G)|= |G|$ (Proposition \ref{noinvolutions}), and then apply Proposition \ref{basefactor}.

We first introduce some set-theoretical notions.
We will work within the Zermelo-Frankel axiomatic system with the Axiom of Choice in the form of the Well-Ordering Theorem. We recall the definition of a well-order.
\begin{Definition}
A well-order $\prec$ on a set $X$ is a total order on $X$ with the property that every nonempty subset of $X$ has a least element.
\end{Definition}

In particular, the following theorem is equivalent to the Axiom of Choice.

\begin{Theorem}[Well-Ordering]
Every set $X$ admits a well-order $\prec$.
\end{Theorem}

Given an element $x\in X$, we define the section $X_{\prec x}$ associated to it
$$X_{\prec x}=\{y\in X: y\prec x\}.$$

\begin{Corollary}[see \cite{Costa20}]\label{goodgoodorder}
Every set $X$ admits a well-order $\prec$ such that the cardinality of any section is smaller than $|X|$.
\end{Corollary}
We recall now that well-orderings allow proofs by induction.

\begin{Theorem}[Transfinite induction]\label{transfinite}
Let $X$ be a set with a well-order $\prec$ and let $P_x$ denote a property for each
$x\in X$. Set $0=\min X$ and assume that:
\begin{itemize}
\item $P_0$ is true, and
\item for every $x\in X$, if $P_y$ holds for every $y\in X_{\prec x}$, then $P_x$ holds.
\end{itemize}
Then $P_x$ is true for every $x\in X$.
\end{Theorem}

We are now ready to build a $G$-regular $1$-factorization of $Cay[G:U]$ whenever $U$ has the same cardinality of $G$ and contains no involutions.

\begin{Proposition}\label{noinvolutions}
Let $\Gamma=Cay[G:U]$ where $G$ is an infinite group.
If $U\,\cap\, I(G) = \emptyset$ and $|U|=G$, then $\Gamma$ has a $G$-regular $1$-factorization.
\end{Proposition}
\begin{proof} We endow $G$ with a well-ordering $\prec$ that satisfies the property of Corollary \ref{goodgoodorder} and such that $0=\min_{\prec} G$. To prove the assertion, it is enough to build, by transfinite induction, an ascending chain of $1$-regular graphs $\Gamma_g$, $g\in G$, (i.e. $\Gamma_h$ is a subgraph of $\Gamma_g$ whenever $h\prec g$) each of which satisfies the following conditions:
\begin{itemize}
\item[$1_g$)] $\Gamma_g$ is either finite or $|V(\Gamma_g)|\leq|G_{\prec g}|$;
\item[$2_g$)] $g\in V(\Gamma_g)$, and $g\in \Delta\Gamma_g$ whenever $g\in U$;
\item[$3_g$)] $\Delta\Gamma_g\subset U$.
\end{itemize}
Indeed, one can easily see that $\Gamma:=\bigcup_{g\in G} \Gamma_g$ is a $1$-factor of $K_G$ such that
$\Delta\Gamma=U$. The result then follows from Proposition \ref{basefactor}.\\

\noindent
BASE CASE. Take $z\in U$ and let $\Gamma_0$ be the edge $\{0,z\}$. This graph clearly satisfies properties $1_0,2_0$ and $3_0$: indeed, $\Gamma_0$ is finite, $0\in V(\Gamma_0)$ and $\Delta \Gamma_0= \{\pm z\}\subset U$.\\

\noindent
INDUCTIVE STEP. Assume that there exists a graph $\Gamma_h$ that satisfies properties $1_h, 2_h$ and $3_h$ for every $h\prec g$, and set $\Gamma_{\prec g}:= \bigcup_{h\prec g} \Gamma_h$. Due to properties $1_h, 2_h$ and $3_h$, $\Gamma_{\prec g}$ is a $1$-regular graph such that:
\begin{itemize}
\item[$1$)] $\Gamma_{\prec g}$ is either finite or $|V(\Gamma_{\prec g})|\leq|G_{\prec g}|$;
\item[$2$)] for every $h\prec g$, $h\in V(\Gamma_{\prec g})$, and $h\in \Delta\Gamma_{\prec g}$
whenever $h \in U$;
\item[$3$)] $\Delta \Gamma_{\prec g}\subset U$.
\end{itemize}
Note that $|\Delta\Gamma_{\prec g}|=|V(\Gamma_{\prec g})|\leq|G_{\prec g}|$, and
$|G_{\prec g}|<|G|$ by Corollary \ref{goodgoodorder}. Therefore, letting
$H = V(\Gamma_{\prec g})\, \cup\, (\Delta\Gamma_{\prec g}+g)$, we have that $|H|<|G|$. Since by assumption $|U|= |G|$, then $|U+g|= |G|$, hence $(U+g) \setminus H$ is nonempty.

Take $z\in (U+g) \setminus H$. Clearly, $z\not \in V(\Gamma_{\prec g})$ and $z-g \in U\setminus \Delta \Gamma_{\prec g}$.
If $g\in V(\Gamma_{\prec g})$, we set $\Gamma'=\Gamma_{\prec g}$, otherwise $\Gamma'$ is obtained by adding to $\Gamma_{\prec g}$ the edge $\{g,z\}$.

Now, if $g\not \in U$ or $g\in \Delta(\Gamma')$, we set $\Gamma_g=\Gamma'$. Otherwise, due to property $1)$ of $\Gamma_{\prec g}$, the set $H' = V(\Gamma')\,\cup\, (-g+V(\Gamma'))$ has cardinality smaller than $|G|$, hence $G\setminus H'$ is nonempty. Then we can take $y\in G\setminus H'$. Clearly, both $y$ and $g+y$ do not belong to $V(\Gamma')$, therefore $\Gamma_g$ is obtained by adding to $\Gamma'$
the edge $\{y,g+y\}$.
\end{proof}

\section{1-Factorizations of complete (equipartite) infinite graphs}
In this section, we prove the main result of this paper, Theorem \ref{main}. We recall that the complete equipartite graph
$K_m[n]$ is isomorphic to the Cayley graph $Cay[G:G\setminus H]$ where $G$ is any group of order $mn$ and $H$ is any subgroup of $G$ of order $n$ whose index in $G$ is $m$.
Because of Theorem \ref{main1}, it is enough to show that $G\setminus (H\,\cup\, I(G))$ has the same cardinality as $G$. This is the content of Theorem \ref{Gruppi} whose proof relies on elementary group theory.

For the reader's convenience, we recall some basic results concerning groups and refer to the
\cite{Ma12} for the standard notions and definitions.

Let $x$ and $g$ be elements of a group $G$. Then $x^g = -g + x + g$ is called the conjugate of $x$ by
$g$, and the set $x^G = \{x^g: g\in G\}$ of all conjugates of $x$ is called the $G$-orbit of $x$.
Note that $x$ and $x^g$ have the same order; also, $x$ and $g$ commute if and only if $x=x^g$.
We recall that the centralizer of $x$ is the subgroup $C(x)$ of $G$
consisting of all group elements that commute with $g$, that is,
\[C(x):=\{g: x^g=x\}.\]
The number $|G:H|$ of right (left) cosets of the subgroup $H$ in $G$ is called the index of $H$ in $G$.
It is very well known that
\[
|G:C(x)| = |x^G|.
\]
Considering that $G$ is the union of all right (resp. left) cosets of $H$, we also have that
\[
|G| = |G:H||H|.
\]
%\begin{Lemma}\label{coniugio}
%For every $h\in G$, the set $\{x: x+ g - x=h\}$ is either empty or it is a left coset of $C(g)$ in $G$.
%\end{Lemma}
%\begin{proof} Set $C_{g,h} =\{x: x+ g - x=h\}$.
%Assume that $C_{g,h}\not=\emptyset$ and take $z\in C_{g,h}$.
%Since $-z\in C_{h,g}$, it is easy to show that $z+C(g) = C_{g,h}$.
%%Let us suppose $\{x: x+ g - x=h\}\not=\emptyset$ and let $z$ be such that $z+g-z=h$. Then, for any $y\in C(g)$ $z+y\in \{x: x+ g - x=h\}$. In fact
%%$$(z+y)+ g - (z+y)= (z+y)+ g - y-z=z+ g - z=h.$$
%%
%%Conversely, given $z,y\in \{x: x+ g - x=h\}$ we have that $h= z+ g - z= y + g- y$ that means $-y+z+ g=g- y+ z.$ It follows that $-y+z\in C(g)$. Therefore $\{x: x+ g - x=h\}$ is a left coset of $C(g)$ in $G$.
%\end{proof}
%
%We start by showing that if an infinite group $G$ does not contain only involutions, then the number of its elements of order greater than 2 equals the cardinality of $G
%$.
We recall that $I(G)$ denotes the set of all elements of $G$ of order 2, also called involutions.

\begin{Lemma}\label{cardinality}
If $G$ is an infinite group and $G\setminus (I(G)\,\cup\,\{0\})$ is nonempty, then
$|G\setminus (I(G)\,\cup\,\{0\})|= |G|$.
\end{Lemma}
\begin{proof}
Suppose that $U = G\setminus (I(G)\,\cup\,\{0\})$ is non\-empty and let
$x\in U$. We assume for a contradiction that $|U|<|G|$.

Set $J = C(x) \,\cap\, I(G)$. Since $x\in U$, then $x+j\in U$ for every $j\in J$
(otherwise, $x+j\in I(G)$ for some $j\in J$, hence $0 = 2(x+j) = 2x + 2j = 2x$,
contradicting the assumption that $x$ is not an involution). In other words,
$x + J \subseteq U$, hence $|J| = |x + J| \leq|U| < |G|$.
Since $C(g) = J \,\cup\,(C(g)\,\cap\,U)$, we have that
\begin{equation}\label{cardinality1}
|C(g)| < |G|.
\end{equation}
Since the conjugacy preserves the order of an element, we have that
$x^G\subseteq U$, hence
\begin{equation}\label{cardinality2}
|G : C(g)| = |x^G| < |G|.
\end{equation}
By conditions \ref{cardinality1} and \ref{cardinality2}, we obtain the following contradiction: $|G|=|G : C(g)|\cdot |C(g)| <|G|^2 = |G|$. Therefore, $|U|=|G|$.
\end{proof}

We can now prove the following result.

\begin{Theorem}\label{Gruppi} Let $G$ be an infinite group and let $H$ be a subgroup of $G$. If $U=G\setminus (I(G)\,\cup\, H)$ is nonempty, then $|U|= |G|$.
\end{Theorem}
\begin{proof} Let $U=G\setminus (I(G)\,\cup\, H)$. By Lemma \ref{cardinality},
$G\setminus(I(G)\,\cup\,\{0\})$ is empty or it has the same cardinality as $G$.
Hence, if $|H|<|G|$, then $|U| = 0$ or $|G|$.

It is left to consider the case $|H| = |G|$.
We assume that $|U|<|H|$ and show that $U$ is necessarily empty.
Note that every right coset of $H$ must contain some involution (otherwise, $U$ would contain a right coset of $H$, which has the same cardinality as $|H|$).
Therefore, denoting by $J = I(G)\,\cap\,(G\setminus H)$ the set of all involutions of $G\setminus H$,
we have that
\begin{equation}\label{cosets}
\text{each right coset of $H$, except for $H$, is of the form $H+j$, with $j\in J$.}
\end{equation}
For each involution $j\in J$, let $H_j\subseteq H$ be the set defined as follows:
\begin{equation}\label{k^i=-k}
\text{$h\in H_j$ if and only if $2(h + j)= 0$.}
\end{equation}
In other words,
\begin{equation}\label{k^i=-k:2}
\text{$h\in H_j$ if and only if $h^j = -h$.}
\end{equation}
Note that $\langle H_j \rangle = H$, where $\langle H_j \rangle$
is the group generated by $H_j$, for every $j\in J$.
Indeed, by \eqref{k^i=-k} and recalling that $|U|< |H|$, we have
\[|H\setminus H_j| = |(H\setminus H_j)+j| = |(H+j)\,\cap\,U|<|H|,\]
hence $|H_j| = |H|$ and $\langle H_j \rangle = H$.

We now show that $H$ is abelian.
Let $j\in J$, $h\in H_j$, and set $U^* = H_j\, \cap\, (U -j - h)$ and
$H_{j,h}=H_j \setminus U^*$.
Clearly, $|U^*|\leq|U| < |H|$,
hence $|H_{j,h}|=|H|$ and
\[\langle H_{j,h}\rangle = H.\]
Also, if $x\in H_{j,h}$, then $(x+h)+j\not\in U\, \cup\, H$, hence $(x+h)+j\in I(G)$.
By \eqref{k^i=-k}, $x+h\in H_j$.
By \eqref{k^i=-k:2}, for every $x\in H_{j,h}\subseteq H_j$ we have that:
\[ -h -x = (x + h)^j = x^j + h^j = -x -h,
\]
that is, $h+x = x+h$.
Then, all the elements of $\langle H_{j,h}\rangle = H$ commute with every $h\in H_j$. This means that the elements of $H_j$ commute with each other, and since they generate $H$, we have that
$H$ is abelian.

Since $H=\langle H_i \rangle$, for every $h\in H$ and $j\in J$ we have that $h^j = -h$, hence
$H=H_j$, and by \eqref{k^i=-k} we have that $H+j$ contains only involutions. Then by \eqref{cosets},
all right cosets of $H$, except for $H$, contain no element of order greater than 2, that is,
$U$ is empty.
\end{proof}

We are now ready to prove the main result of this paper, whose statement is recalled in the following.\\

\noindent
\textbf{Theorem \ref{main}.}
\emph{
Let $G$ be an infinite group. There exists a $G$-regular $1$-factorization of
$K_m[n]$ if and only if $G$ has a subgroup $H$ of size $n$ whose index in $G$ is $m$.
}
\begin{proof}
Let $G$ be an infinite group and let $H$ be a subgroup of $G$ of size $n$ whose index in $G$ is $m$.
By Theorem \ref{Gruppi}, the set $G\setminus (I(G)\cup H)$ is either empty or has the same size as $G$. Therefore, by applying Theorem \ref{main1} with $S=G\setminus H$, we obtain the existence of a $G$-regular $1$-factorization of $Cay[G:S]$. Clearly, $Cay[G:S]$ is isomorphic to $K_m[n]$.

Conversely, assume there is a $G$-regular $1$-factorization $\mathcal{G}$ of $K_m[n]$.
By Corollary \ref{Sabidussi:2}, $K_m[n]$ is isomorphic to $Cay[G:S]$ for some connection set $S$ of $G$.
Considering that $Cay[G:S]$ contains no edge of the form $\{0,h\}$ for every $h\in H=G\setminus S$,
it follows that $H$ represents a part (of size $n$) of the equipartite complete graph $Cay[G:S]$.
We are going to prove that $H$ is a subgroup of $G$. If $x,y\in H$ and $y-x\in S$,
then $\{x,y\}$ would be an edge of $Cay[G:S]$, contradicting the fact that $H$ is a part of $Cay[G:S]$.
Therefore, $y-x\in H$, for every $x,y\in H$, that is, $H$ is a subgroup of $G$.
\end{proof}

By taking $n=1$ in Theorem \ref{main}, we obtain the following corollary.
\begin{Corollary}
There exists a $G$-regular $1$-factorization of $K_G$ for every infinite group $G$.
\end{Corollary}

\section{$1$-factorizations with subfactorizations}
In this section, given an $H$-regular $1$-factorization 
$\mathcal{H}$ of $K_{m'}[n']$, and a group $G$ containing $H$,
we provide conditions on $G,m,$ and $n$ that guarantee
the existence of a $G$-regular $1$-factorization $\mathcal{G}$ of $K_m[n]$ that contains $\mathcal{H}$ as a subfactorization.
This means that for every pair of $1$-factors $(F,\Gamma)\in \mathcal{H}\times \mathcal{G}$,
either $F\subseteq \Gamma$ or $F\,\cap\, \Gamma$ is empty.
When speaking of an $H$-regular subfactorization of $\mathcal{G}$, it is understood that both $G$ and $H$ act on the related
$1$-factorizations by right translation.

Given two cardinals, $m'$ and $m$, we write $m'|m$ whenever 
$m$ is infinite and $m'\leq m$, or $m$ is finite and $m'$ 
is a divisor of $m$ . In the former case, we set $m/m'=m$. We notice that, similarly to the finite case, we have that $(m/m')\cdot m' = m$.
This convention allows us to consider the case where one parameter between $m$ and $n$ (which define the equipartite complete graph $K_{m}[n]$) is finite.

To ease the notation, given a direct product of groups 
$G=G_1\times H$,  we denote by $G_1$ and $H$
the subgroups $G_1\times \{0\}$ and $\{0\}\times H$ of $G$, respectively. In other words, we consider $G$ as the direct inner product of its two trivially intersecting subgroups 
$G_1$ and $H$.

\begin{Lemma} \label{nested:lemma}
Let $\mathcal{H}$ be an $H$-regular $1$-factorization of $Cay[H:H\setminus K]$, and set $G=G_1\times H$ 
for some group $G_1$.
Then there exists a $G$-regular $1$-factorization of $Cay[G:H\setminus K]$ containining $\mathcal{H}$ as a subfactorization.
\end{Lemma}
\begin{proof}
Let $\mathcal{H}^*=\{F^*: F\in\mathcal{H}\}$ be the set of $1$-factors of $K_G$ obtained from those in $\mathcal{H}$ as follows:
\[
F^* = \bigcup_{x\in G_1}(F + x).
\]
Clearly, $\mathcal{H}^*$ is a $1$-factorization of $Cay[G:H\setminus K]$. To prove that it is $G$-regular, 
it is enough to check that, for every $F^*\in \mathcal{H}^*$, 
$g\in G_1$ and $h\in H$, $F^*+(g+h)\in \mathcal{H}^*$.
Note that
\[
F^*+g+h=\bigcup_{x\in G_1}(F + x)+(g+h)=
\bigcup_{g'\in G_1}(F + g'+h).
\]
Recalling that $\mathcal{H}$ is regular under the action of $H$ by right translation, we have that 
$F+h=F'\in \mathcal{H}$ which implies
\[
 \bigcup_{g'\in G_1}(F + g'+h)
=\bigcup_{g'\in G_1}(F' + g')=(F')^{*}\in\mathcal{H}^*.
\]
The assertion follows.
\end{proof}

%\begin{Definition}
%A $1$-factorization $\mathcal{G}$ of $\Gamma$ is said to be a nested subsystem of a $1$-factorization $\mathcal{F}$ of $\Gamma'$ if $\Gamma$ is a subgraph of $\Gamma'$ and $\mathcal{G}$ is the restriction of $\mathcal{F}$ on $\Gamma'$.
%\end{Definition}
%Using this terminology, we want to determine some conditions under which there exists a regular $1$-factorization of $K_{m\times n}$ that admits a given regular $1$-factorization of $K_{\ell\times n}$ as a nested subsystem.
%Since, due to Remark \ref{RemCay}, any complete equipartite graph that admits a $G$-regular $1$-factorization is isomorphic to $Cay[G: (G\setminus K)]$ for some subgroup $K$ of $G$, we find more convenient to study this problem on Cayley graphs. In the following, we will use the convention to say that $\ell|m$ whenever $m$ is infinite and $m\geq \ell$ and when $m$ is finite and $\ell|m$ as an integer.

\begin{Theorem}\label{Nested}
Let $\mathcal{H}$ be an $H$-regular $1$-factorization of $K_{m'}[n']$. Also, let $m$ and $n$ be cardinals such that $mn$ is infinite, $m'|m$ and $n'|n$. Then, there exists a regular $1$-factorization of $K_{m}[n]$ containing $\mathcal{H}$ as a subfactorization.
\end{Theorem}
\begin{proof}
Let $\mathcal{H}$ be a nonempty $H$-regular $1$-factorization of $K_{m'}[n']$.  
Up to isomorphism, we can assume that $K_{m'}[n'] = Cay[H:H\setminus K]$
where $|H|=m'n'$ and $|K|=n'$, and that $\mathcal{H}$ is $H$-regular 
under the action by right translation, that is, for every $F\in \mathcal{H}$
and $h\in H$, we have that $F+h\in\mathcal{H}$. 

Let $G_1$ and $L_1$ be groups of order $n/n'$ and $m/m'$, respectively.
Also,  set $L=L_1\times K$ and $G = G_1\times L_1\times H$.
Since one between $m$ and $n$ is infinite, then 
\[
|G_1\times L_1|=(m/m')(n/n')=\max(m/m',n/n')=\max(m,n)=mn=|G|. 
\]
Denoting by $U$ the set of non-involutions  of $G_1\times L_1$ and assuming 
that $|U|>0$, by Lemma \ref{cardinality}
we have that $|U|=|G_1\times L_1|=|G|$. Therefore, $U\times (H\setminus K)$
is a set of non-involutions, 
of cardinality $|G|$, contained in $G\setminus(H\, \cup \, L)$.
Similarly, if $U$ denotes the set of non-involutions of $H$ and $|U|>0$, 
then  $((G_1\times L_1)\setminus\{0\}) \times U$ has the same cardinality as $G$, and it contains only elements of order greater than $2$ belonging to 
$G\setminus(H\,\cup\, L)$.
Therefore, 
if $G\setminus (H\,\cup\, L)$ contains some non-involutions, then the number of its elements of order greater than $2$ is $|G|$. Hence, by Theorem \ref{main1}, there is a $G$-regular $1$-factorization $\mathcal{F}_1$ of $Cay[G:G\setminus (H\,\cup\, L)]$.
Moreover, due to Lemma \ref{nested:lemma}, there also exists a $G$-regular $1$-factorization $\mathcal{F}_2$ of $Cay[G:H\setminus  K]$. Considering that $G\setminus( H\,\cup\, L)$ and $H\setminus  L =  H\setminus K$ partition $G\setminus L$, it follows that $\mathcal{F}_1\,\cup\,\mathcal{F}_2$ is a $G$-regular $1$-factorization of $Cay[G:G\setminus L]=K_{m}[n]$ containing 
$\mathcal{H}$ as a subfactorization.
\end{proof}

As a corollary, we obtain the following.

\begin{Corollary}
Let $\mathcal{H}$ be a regular $1$-factorization of $K_{m'}$. Then, given an infinite cardinal $m$, there exists a regular $1$-factorization of $K_{m}$ that admits $\mathcal{H}$ as a subfactorization if and only if $m'|m$.
\end{Corollary}
%\begin{proof}
%Let $\mathcal{G}$ be regular with respect to the infinite group $H$ and let us assume that there exists a $G$ regular $1$-factorization of $K_{\aleph}$ that admits $\mathcal{G}$ as a nested subsystem.
%Then, due to Remark \ref{RemCay}, $K_{\aleph'}$ is isomorphic to $Cay[H:(H\setminus \{0\})]$ and $K_{\aleph}$ is isomorphic to $Cay[G:(G\setminus \{0\})]$ where $H$ is a subgroup of $G$. It follows that $\aleph=|G|$ divides $|H|=\aleph'$.
%
%On the other hand, let us assume that $\aleph'|\aleph$. Then, according to Theorem \ref{Nested}, there exists a group $G$ of size $\aleph$ and a $G$-regular $1$-factorization of $Cay[G:(G\setminus \{0\})]$ that extends $\mathcal{G}$.
%\end{proof}

\section*{Acknowledgements}
The authors were partially supported by INdAM--GNSAGA.

\end{document}